\documentclass[a4paper,12pt]{article}
\usepackage[a4paper,margin=2.5truecm]{geometry}
\usepackage{amsmath,amsfonts,amssymb,amsthm,txfonts,hyperref,url}
\newtheorem{lemma}{Lemma}
\newtheorem{corollary}{Corollary}
\newtheorem{theorem}{Theorem}
\newcommand{\N}{\mathbb N}
\newcommand{\Z}{\mathbb Z}
\title{Is there an algorithm that decides the solvability
of a Diophantine equation with a finite number of solutions?}
\author{Apoloniusz Tyszka}
\begin{document}
\begin{sloppypar}
\date{}
\maketitle
\begin{abstract}
For a positive integer $n$, let \mbox{$\theta(n)$} denote the smallest
 positive integer $b$
such that for each system \mbox{$S \subseteq \{x_i \cdot x_j=x_k,~x_i+1=x_k:~i,j,k \in \{1,\ldots,n\}\}$}
which has
 a solution in positive integers \mbox{$x_1,\ldots,x_n$} and which has only finitely many solutions
in positive integers \mbox{$x_1,\ldots,x_n$}, 
there exists a solution of $S$ in \mbox{$([1,b] \cap \N)^n$}.
We conjecture that there exists an integer \mbox{$\delta \geqslant 9$} such that the inequality
\mbox{$\theta(n) \leqslant \left(2^{\textstyle 2^{n-5}}-1\right)^{\textstyle 2^{n-5}}+1$}
holds for every integer \mbox{$n \geqslant \delta$}. We prove:
{\tt (1)} for every integer \mbox{$n>9$}, the inequality
\mbox{$\theta(n)<\left(2^{\textstyle 2^{n-5}}-1\right)^{\textstyle 2^{n-5}}+1$} implies that
\mbox{$2^{\textstyle 2^{n-5}}+1$} is composite,
{\tt (2)}~the conjecture implies that there exists
an algorithm which takes as input a Diophantine equation \mbox{$D(x_1,\ldots,x_p)=0$} and returns
the message {\tt "Yes"} or {\tt "No"} which correctly determines the solvability of the equation \mbox{$D(x_1,\ldots,x_p)=0$}
in positive integers, if the solution set is finite,
{\tt (3)} if a function \mbox{$f \colon \N \setminus \{0\} \to \N \setminus \{0\}$} has a \mbox{finite-fold}
Diophantine representation, then there exists a positive integer $m$ such that \mbox{$f(n)<\theta(n)$}
for every integer \mbox{$n>m$}.
\end{abstract}
\vskip 1.0truecm
\par
In this article, we propose a conjecture which implies that there exists an algorithm which takes as input
a Diophantine equation \mbox{$D(x_1,\ldots,x_p)=0$} and returns the message {\tt "Yes"} or {\tt "No"} which
correctly determines the solvability of the equation \mbox{$D(x_1,\ldots,x_p)=0$} in positive integers,
if the solution set is finite. Let
\[
E_n=\{x_i \cdot x_j=x_k,~x_i+1=x_k:~i,j,k \in \{1,\ldots,n\}\}
\]
For a positive integer $n$, let \mbox{$\theta(n)$} denote the smallest
 positive integer $b$
such that for each system \mbox{$S \subseteq E_n$} which has
 a solution in positive integers
\mbox{$x_1,\ldots,x_n$} and which has only finitely many solutions in positive integers
\mbox{$x_1,\ldots,x_n$}, there exists a solution of $S$ in \mbox{$([1,b] \cap \N)^n$}.
We do not know whether or not there exists a computable function
\mbox{$\xi \colon \N \setminus \{0\} \to \N \setminus \{0\}$} which is greater
than the function \mbox{$\theta \colon \N \setminus \{0\} \to \N \setminus \{0\}$}.
\begin{theorem}\label{the1}
We have: \mbox{$\theta(1)=1$} and \mbox{$\theta(2)=2$}.
The inequality \mbox{$\theta(n) \geqslant 2^{\textstyle 2^{n-2}}$}
holds for every integer \mbox{$n \geqslant 3$}.
\end{theorem}
\begin{proof}
Only \mbox{$x_1=1$} solves the equation \mbox{$x_1 \cdot x_1=x_1$} in positive integers.
Only \mbox{$x_1=1$} and \mbox{$x_2=2$} solve the system \mbox{$\{x_1 \cdot x_1=x_1,~x_1+1=x_2\}$}
in positive integers. For each integer \mbox{$n \geqslant 3$}, the following system
\begin{displaymath}
\left\{
\begin{array}{rcl}
x_1 \cdot x_1 &=& x_1 \\
x_1+1 &=& x_2 \\
\forall i \in \{2,\ldots,n-1\} ~x_i \cdot x_i &=& x_{i+1}
\end{array}
\right.
\end{displaymath}
\noindent
has a unique solution in positive integers,
namely \mbox{$\left(1,2,4,16,256,\ldots,2^{\textstyle 2^{n-3}},2^{\textstyle 2^{n-2}}\right)$}.
\end{proof}
\begin{theorem}\label{the2}
For every positive integer $n$, \mbox{$\theta(n+1) \geqslant \theta(n) \cdot \theta(n)$}.
\end{theorem}
\begin{proof}
For every \mbox{$k \in \{1,\ldots,n\}$}, if a system \mbox{$S \subseteq E_n$} has
only finitely many solutions in positive integers \mbox{$x_1,\ldots,x_n$}, then the
system \mbox{$S \cup \{x_k \cdot x_k=x_{n+1}\}$} has only finitely many solutions
in positive integers \mbox{$x_1,\ldots,x_{n+1}$}.
\end{proof}
\begin{corollary}\label{cor1}
\mbox{$\theta(1)=1<2=\theta(2)<\theta(3)<\theta(4)<\ldots$}
\end{corollary}
\par
Primes of the form \mbox{$2^{\textstyle 2^n}+1$} are called Fermat primes,
as Fermat conjectured that every integer of the form \mbox{$2^{\textstyle 2^n}+1$} is prime (\mbox{\cite[p.~1]{17lectures}}).
Fermat correctly remarked that \mbox{$2^{\textstyle 2^0}+1=3$}, \mbox{$2^{\textstyle 2^1}+1=5$},
\mbox{$2^{\textstyle 2^2}+1=17$}, \mbox{$2^{\textstyle 2^3}+1=257$}, and \mbox{$2^{\textstyle 2^4}+1=65537$}
are all prime (\mbox{\cite[p.~1]{17lectures}}).
\vskip 0.2truecm
\noindent
{\bf Open Problem.} {\em Are there infinitely many composite numbers of the
form \mbox{$2^{\textstyle 2^n}+1$}? (\mbox{\cite[p.~159]{17lectures}})}
\vskip 0.2truecm
\noindent
Most mathematicians believe that \mbox{$2^{\textstyle 2^n}+1$} is composite for every integer \mbox{$n \geqslant 5$}.
\begin{theorem}\label{the3}
If $n \in {\mathbb N}\setminus\{0\}$ and $2^{\textstyle 2^n}+1$ is prime,
then the following system
\[
\left\{\begin{array}{rcl}
\forall i \in \{1,\ldots,n\} ~x_i \cdot x_i &=& x_{i+1} \\
x_{1}+1 &=& x_{n+2} \\
x_{n+2}+1 &=& x_{n+3} \\
x_{n+1}+1 &=& x_{n+4} \\
x_{n+3} \cdot x_{n+5} &=& x_{n+4}
\end{array}\right.
\]
has a unique solution \mbox{$\left(a_1,\ldots,a_{n+5}\right)$} in \mbox{non-negative} integers.
The numbers \mbox{$a_1,\ldots,a_{n+5}$} are positive and
${\rm max}\left(a_1,\ldots,a_{n+5}\right)=a_{n+4}=\left(2^{\textstyle 2^n}-1\right)^{\textstyle 2^n}+1$.
\end{theorem}
\begin{proof}
The system equivalently expresses that \mbox{$x_1^{\textstyle 2^n}+1=(x_1+2) \cdot x_{n+5}$}. Therefore,
\[
x_1^{\textstyle 2^n}+1=\left((x_1+2)-2\right)^{\textstyle 2^n}+1=
\]
\[
2^{\textstyle 2^n}+1+(x_1+2) \cdot \sum_{\textstyle k=1}^{\textstyle 2^n}
{\textstyle 2^n \choose k} \cdot (x_1+2)^{\textstyle k-1} \cdot (-2)^{\textstyle 2^n-k}=(x_1+2) \cdot x_{n+5}
\]
Hence,
\[
2^{\textstyle 2^n}+1=
(x_1+2) \cdot \left(x_{n+5}-\sum_{\textstyle k=1}^{\textstyle 2^n} {{\textstyle 2^n} \choose {\textstyle k}}
\cdot (x_1+2)^{\textstyle k-1} \cdot (-2)^{\textstyle 2^n-k}\right)
\]
Therefore, \mbox{$x_1+2$} divides \mbox{$2^{\textstyle 2^n}+1$}.
Since \mbox{$x_1+2 \geqslant 2$} and \mbox{$2^{\textstyle 2^n}+1$} is prime, we get
\mbox{$x_1=2^{\textstyle 2^n}-1$}. Hence, \mbox{$x_{n+2}=2^{\textstyle 2^n}$}
and \mbox{$x_{n+3}=2^{\textstyle 2^n}+1$}. Next,
\mbox{$x_{n+1}=x_1^{\textstyle 2^n}=\left(2^{\textstyle 2^n}-1\right)^{\textstyle 2^n}$} and
\[
x_{n+4}=x_{n+1}+1=\left(2^{\textstyle 2^n}-1\right)^{\textstyle 2^n}+1
\]
The following positive integers
\[
\begin{array}{rcl}
\forall i \in \{1,\ldots,n+1\} ~a_i &=& \left(2^{\textstyle 2^n}-1\right)^{\textstyle 2^{i-1}} \\
a_{n+2} &=& 2^{\textstyle 2^n} \\
a_{n+3} &=& 2^{\textstyle 2^n}+1 \\
a_{n+4} &=& \left(2^{\textstyle 2^n}-1\right)^{\textstyle 2^n}+1 \\
a_{n+5} &=& 1+\sum_{\textstyle k=1}^{\textstyle 2^{n}}\limits
\displaystyle {{\textstyle 2^n} \choose {\textstyle k}} \cdot \left(2^{\textstyle 2^n}+1\right)^{\textstyle k-1}
\cdot (-2)^{\textstyle 2^n-k}
\end{array}
\]
give the solution which is unique in \mbox{non-negative} integers. The number $a_{n+5}$ is positive because
\[
a_{n+5}=\frac{a_{n+4}}{a_{n+3}}=\frac{\left(2^{\textstyle 2^n}-1\right)^{\textstyle 2^n}+1}{2^{\textstyle 2^n}+1}
\]
\end{proof}
\begin{corollary}\label{cor2}
For every integer \mbox{$n>5$}, if \mbox{$2^{\textstyle 2^{n-5}}+1$} is prime, then
\[
\theta(n) \geqslant \left(2^{\textstyle 2^{n-5}}-1\right)^{\textstyle 2^{n-5}}+1
\]
In particular, 
\[
\theta(9) \geqslant \left(2^{\textstyle 2^{9-5}}-1\right)^{\textstyle 2^{9-5}}+1=\left(2^{16}-1\right)^{16}+1>
\left(2^{16}-2^{15}\right)^{16}=\left(2^{15}\right)^{16}=2^{240}>2^{\textstyle 2^{9-2}}
\]
The numbers \mbox{$2^{\textstyle 2^{n-5}}+1$} are prime when \mbox{$n \in \{6,7,8\}$}, but
\[
\left(2^{\textstyle 2^{6-5}}-1\right)^{\textstyle 2^{6-5}}+1=10<65536=2^{\textstyle 2^{6-2}}
\]
\[
\left(2^{\textstyle 2^{7-5}}-1\right)^{\textstyle 2^{7-5}}+1=50626<4294967296=2^{\textstyle 2^{7-2}}
\]
\[
\left(2^{\textstyle 2^{8-5}}-1\right)^{\textstyle 2^{8-5}}+1=17878103347812890626<18446744073709551616=2^{\textstyle 2^{8-2}}
\]
\end{corollary}
\begin{corollary}\label{cor3}
For every integer \mbox{$n>9$}, the inequality \mbox{$\theta(n)<\left(2^{\textstyle 2^{n-5}}-1\right)^{\textstyle 2^{n-5}}+1$}
implies that \mbox{$2^{\textstyle 2^{n-5}}+1$} is composite.
\end{corollary}
\vskip 0.01truecm
\noindent
{\bf Conjecture.} (cf. \cite[p.~710]{Tyszka1})
{\em There exists an integer \mbox{$\delta \geqslant 9$} such that the inequality
\[
\theta(n) \leqslant \left(2^{\textstyle 2^{n-5}}-1\right)^{\textstyle 2^{n-5}}+1
\]
holds for every integer \mbox{$n \geqslant \delta$}.}
\begin{corollary}\label{cor4}
By Corollary~\ref{cor1}, the Conjecture implies that there exists a computable function
\mbox{$\xi \colon \N \setminus \{0\} \to \N \setminus \{0\}$} which is greater
than the function \mbox{$\theta \colon \N \setminus \{0\} \to \N \setminus \{0\}$}.
\end{corollary}
\par
Let $\alpha$, $\beta$, and $\gamma$ denote variables.
\begin{lemma}\label{lem1} (\cite[p.~100]{Robinson})
For each positive integers \mbox{$x,y,z$}, \mbox{$x+y=z$} if and only if
\[
(zx+1)(zy+1)=z^2(xy+1)+1
\]
\end{lemma}
\begin{corollary}\label{cor5}
We can express the equation \mbox{$x+y=z$} as an equivalent system ${\cal F}$,
where ${\cal F}$ involves \mbox{$x,y,z$} and $9$ new variables, and where ${\cal F}$ consists of equations
of the forms \mbox{$\alpha+1=\gamma$} and \mbox{$\alpha \cdot \beta=\gamma$}.
\end{corollary}
\begin{proof}
The new $9$ variables express the following polynomials:
\[
zx,~~~~~zx+1,~~~~~zy,~~~~~zy+1,~~~~~z^2,~~~~~xy,~~~~~xy+1,~~~~~~z^2(xy+1),~~~~~z^2(xy+1)+1
\]
\end{proof}
\begin{lemma}\label{lem2}
Let \mbox{$D(x_1,\ldots,x_p) \in {\Z}[x_1,\ldots,x_p]$}.
Assume that \mbox{${\rm deg}(D,x_i) \geqslant 1$} for each \mbox{$i \in \{1,\ldots,p\}$}. We can compute a positive
integer \mbox{$n>p$} and a system \mbox{${\cal T} \subseteq E_n$} which satisfies the following two conditions:
\vskip 0.2truecm
\noindent
{\tt Condition 1.} For every positive integers \mbox{$\tilde{x}_1,\ldots,\tilde{x}_p$},
\[
D(\tilde{x}_1,\ldots,\tilde{x}_p)=0 \Longleftrightarrow
\exists \tilde{x}_{p+1},\ldots,\tilde{x}_n \in \N \setminus \{0\} ~(\tilde{x}_1,\ldots,\tilde{x}_p,\tilde{x}_{p+1},\ldots,\tilde{x}_n) ~solves~ {\cal T}
\]
{\tt Condition 2.} If positive integers \mbox{$\tilde{x}_1,\ldots,\tilde{x}_p$} satisfy
\mbox{$D(\tilde{x}_1,\ldots,\tilde{x}_p)=0$}, then there exists a unique tuple
\mbox{$(\tilde{x}_{p+1},\ldots,\tilde{x}_n) \in (\N \setminus \{0\})^{n-p}$} such that the tuple
\mbox{$(\tilde{x}_1,\ldots,\tilde{x}_p,\tilde{x}_{p+1},\ldots,\tilde{x}_n)$} solves ${\cal T}$.
\vskip 0.2truecm
\noindent
Conditions 1 and 2 imply that the equation \mbox{$D(x_1,\ldots,x_p)=0$} and the system ${\cal T}$ have
the same number of solutions in positive integers.
\end{lemma}
\begin{proof}
We write down the polynomial \mbox{$D(x_1,\ldots,x_p)$} and replace each coefficient by the successor
of its absolute value. Let \mbox{$\widetilde{D}(x_1,\ldots,x_p)$} denote the obtained polynomial.
The polynomials \mbox{$D(x_1,\ldots,x_p)+\widetilde{D}(x_1,\ldots,x_p)$} and \mbox{$\widetilde{D}(x_1,\ldots,x_p)$}
have positive integer coefficients. The equation \mbox{$D(x_1,\ldots,x_p)=0$} is equivalent to
\[
D(x_1,\ldots,x_p)+\widetilde{D}(x_1,\ldots,x_p)+1=\widetilde{D}(x_1,\ldots,x_p)+1
\]
There exist positive integers $a$ and $b$ and finite \mbox{non-empty} lists $A$ and $B$ such that
the above equation is equivalent to
\[
\Bigl(\Bigl(\Bigl(\sum_{\textstyle (i_1,j_1,\ldots,i_k,j_k) \in A}
x_{\textstyle i_1}^{\textstyle j_1} \cdot \ldots \cdot x_{\textstyle i_k}^{\textstyle j_k}\Bigr)+
\underbrace{1\Bigr)+\ldots\Bigr)+1}_{\textstyle a~{\rm units}}=
\]
\[
\Bigl(\Bigl(\Bigl(\sum_{\textstyle (i_1,j_1,\ldots,i_k,j_k) \in B}
x_{\textstyle i_1}^{\textstyle j_1} \cdot \ldots \cdot x_{\textstyle i_k}^{\textstyle j_k}\Bigr)+
\underbrace{1\Bigr)+\ldots\Bigr)+1}_{\textstyle b~{\rm units}}
\]
and all the numbers \mbox{$k,i_1,j_1,\ldots,i_k,j_k$} belong to \mbox{$\N \setminus \{0\}$}.
Next, we apply Corollary~\ref{cor5}.
\end{proof}
\begin{theorem}\label{the4}
The Conjecture implies that there exists
an algorithm which takes as input a Diophantine equation \mbox{$D(x_1,\ldots,x_p)=0$} and returns
the message {\tt "Yes"} or {\tt "No"} which correctly determines the solvability of the equation \mbox{$D(x_1,\ldots,x_p)=0$}
in positive integers, if the solution set is finite.
\end{theorem}
\begin{proof}
We apply Lemma~\ref{lem2} and compute the system \mbox{${\cal T} \subseteq E_n$}, where \mbox{$n>p$}.
Let \mbox{$w={\rm max}(n,\delta)$}. By Corollary~\ref{cor1}, it suffices to check whether or not the system ${\cal T}$
has a solution in positive integers
 \mbox{$x_1,\ldots,x_n$} not greater than
\mbox{$\left(2^{\textstyle 2^{w-5}}-1\right)^{\textstyle 2^{w-5}}+1$}.
\end{proof}
\par
The Davis-Putnam-Robinson-Matiyasevich theorem states that every recursively
enumerable set \mbox{${\cal M} \subseteq {\N}^n$} has a Diophantine
representation, that is
\[
(a_1,\ldots,a_n) \in {\cal M} \Longleftrightarrow \exists x_1, \ldots, x_m \in \N ~~W(a_1,\ldots,a_n,x_1,\ldots,x_m)=0 \tag*{\texttt{(R)}}
\]
for some polynomial $W$ with integer coefficients, see \cite{Matiyasevich1}.
The polynomial~$W$ can be computed, if we know the Turing \mbox{machine $M$} such
that, for all \mbox{$(a_1,\ldots,a_n) \in {\N}^n$}, $M$ halts on \mbox{$(a_1,\ldots,a_n)$} if
and only if \mbox{$(a_1,\ldots,a_n) \in {\cal M}$}, \mbox{see \cite{Matiyasevich1}}.
The representation~\texttt{(R)} is said to be \mbox{single-fold}, if for every
\mbox{$a_1,\ldots,a_n \in \N$} the equation
\mbox{$W(a_1,\ldots,a_n,x_1,\ldots,x_m)=0$} has at most one solution
\mbox{$(x_1,\ldots,x_m) \in {\N}^m$}.
The representation~\texttt{(R)} is said to be \mbox{finite-fold}, if for every
\mbox{$a_1,\ldots,a_n \in \N$} the equation
\mbox{$W(a_1,\ldots,a_n,x_1,\ldots,x_m)=0$} has only finitely many solutions
\mbox{$(x_1,\ldots,x_m) \in {\N}^m$}. \mbox{Yu. Matiyasevich} conjectured that
each recursively enumerable set \mbox{${\cal M} \subseteq {\N}^n$} has a
\mbox{single-fold} (\mbox{finite-fold}) Diophantine representation, see \mbox{\cite[pp.~341--342]{DMR}} and
\mbox{\cite[p.~42]{Matiyasevich2}}. Currently, he seems agnostic on his conjectures,
see \mbox{\cite[p.~749]{Matiyasevich3}}.
In \mbox{\cite[p.~581]{Tyszka2}}, the author explains why Matiyasevich's conjectures although widely known are less widely accepted.
\begin{theorem}\label{the5}
If a function \mbox{$f \colon \N \setminus \{0\} \to \N \setminus \{0\}$}
has a \mbox{finite-fold} Diophantine representation,
then there exists a positive integer $m$ such that \mbox{$f(n)<\theta(n)$} for every integer \mbox{$n>m$}.
\end{theorem}
\begin{proof}
There exists a polynomial \mbox{$W(x_1,x_2,x_3,\ldots,x_r)$} with integer coefficients
such that for each positive integers \mbox{$x_1,x_2$},
\[
(x_1,x_2) \in f \Longleftrightarrow \exists x_3,\ldots,x_r \in \N \setminus \{0\}~~ W(x_1,x_2,x_3-1,\ldots,x_r-1)=0
\]
and for each positive integers \mbox{$x_1,x_2$} at most finitely many tuples
\mbox{$(x_3,\ldots,x_r)$} of positive integers satisfy \mbox{$W(x_1,x_2,x_3-1,\ldots,x_r-1)=0$}.
By Lemma~\ref{lem2}, there exists an integer \mbox{$s \geqslant 3$} such that
for every positive integers \mbox{$x_1,x_2$},
\begin{equation}
(x_1,x_2) \in f \Longleftrightarrow \exists x_3,\ldots,x_s \in \N \setminus \{0\}~~ \Phi(x_1,x_2,x_3,\ldots,x_s)\tag*{\texttt{(E)}}
\end{equation}
where \mbox{$\Phi(x_1,x_2,x_3,\ldots,x_s)$} is a conjunction of formulae of the forms
\mbox{$x_i+1=x_k$} and \mbox{$x_i \cdot x_j=x_k$}, the indices $i,j,k$ belong to
$\{1,\ldots,s\}$, and for each positive integers \mbox{$x_1,x_2$} at most finitely many
tuples \mbox{$(x_3,\ldots,x_s)$} of positive integers satisfy \mbox{$\Phi(x_1,x_2,x_3,\ldots,x_s)$}.
Let $[\cdot]$ denote the integer part function, and let an integer $n$ be greater than \mbox{$m=2s+2$}.
Then,
\[
n \geqslant \left[\frac{n}{2}\right]+\frac{n}{2}>\left[\frac{n}{2}\right]+s+1
\]
and \mbox{$n-\left[\frac{n}{2}\right]-s-2 \geqslant 0$}.
Let $T_n$ denote the following system with $n$ variables:
\[
\left\{
\begin{array}{c}
\textrm{all~equations~occurring~in~}\Phi(x_1,x_2,x_3,\ldots,x_s)\\
\begin{array}{rcl}
\forall i \in \left\{1,\ldots,n-\left[\frac{n}{2}\right]-s-2\right\} ~u_i \cdot u_i &=& u_i\\
t_1 \cdot t_1 &=& t_1\\
\forall i \in \left\{1,\ldots,\left[\frac{n}{2}\right]-1\right\} ~t_i+1 &=& t_{i+1}\\
t_2 \cdot t_{\left[\frac{n}{2}\right]} &=& u\\
u+1 &=&x_1 {\rm ~(if~}n{\rm ~is~odd)}\\
t_1 \cdot u &=& x_1 {\rm ~(if~}n{\rm ~is~even)}\\
x_2+1 &=& y
\end{array}
\end{array}
\right.
\]
By the equivalence~\texttt{(E)}, the system \mbox{$T_n$} is solvable in positive integers,
\mbox{$2 \cdot \left[\frac{n}{2}\right]=u$}, \mbox{$n=x_1$}, and
\[
f(n)=f(x_1)=x_2<x_2+1=y
\]
The system $T_n$ consists of equations of the forms \mbox{$\alpha+1=\gamma$} and \mbox{$\alpha \cdot \beta=\gamma$}.
Since \mbox{$T_n$} has only finitely many solutions in positive integers, \mbox{$y \leqslant \theta(n)$}.
Hence, \mbox{$f(n)<\theta(n)$}.
\end{proof}
\begin{corollary}\label{cor6}
The Conjecture contradicts Matiyasevich's conjecture on \mbox{finite-fold} Diophantine representations.
\end{corollary}

\noindent
Apoloniusz Tyszka\\
Technical Faculty\\
Hugo Ko{\l}{\l}\c{a}taj University\\
Balicka 116B, 30-149 Krak\'ow, Poland\\
E-mail: \url{rttyszka@cyf-kr.edu.pl}
\end{sloppypar}
\end{document}